\documentclass[11pt, reqno]{amsart}
\usepackage{amssymb,latexsym,amsmath,amsfonts}
\usepackage[mathscr]{eucal}
\usepackage[stable]{footmisc}

\numberwithin{equation}{section}

\theoremstyle{definition}
\newtheorem{definition}{Definition}[section]

\newtheorem{remark}[definition]{Remark}

\theoremstyle{plain}
\newtheorem{theorem}[definition]{Theorem}
\newtheorem{lemma}[definition]{Lemma}

\newtheorem{result}[definition]{Result}

\newcommand{\eps}{\varepsilon}

\newcommand{\zt}{\zeta}

\newcommand{\al}{\alpha}
\newcommand{\lm}{\lambda}

\newcommand\ba[1]{\overline{#1}}
\newcommand\hull[1]{\widehat{#1}}
\newcommand\wtil[1]{\widetilde{#1}}
\newcommand{\id}{\mathbb{I}}

\newcommand{\sm}{\sigma}

\newcommand{\scr}[1]{\mathscr{#1}}

\newcommand{\bdy}{\partial}
\newcommand{\OM}{\Omega}

\newcommand{\ball}{\mathbb{B}}

\newcommand{\poly}{\mathscr{P}}

\newcommand{\impl}{\Longrightarrow}

\newcommand{\tends}{\rightarrow}

\newcommand{\CC}{\mathbb{C}^2}
\newcommand{\cplx}{\mathbb{C}}
\newcommand{\RR}{\mathbb{R}^2}
\newcommand{\rea}{\mathbb{R}}

\begin{document}

\title[Polynomial convexity of union of totally-real planes]{On the polynomial convexity
of the union of more\\ than two totally-real planes in $\CC$}
\author{Sushil Gorai}
\address{Department of Mathematics, Indian Institute of Science, Bangalore -- 560 012}
\email{sushil.gorai@gmail.com}

\begin{abstract}
 In this paper we shall discuss local polynomial convexity at the origin of
the union of finitely many totally-real planes through $0 \in\cplx^2$. The planes, say $P_0,
\dots, P_N$, satisfy a mild transversality condition that enables us to view them
in Weinstock normal form, i.e. $P_0=\RR$ and $P_j=M(A_j):=(A_j+i\id)\RR$,
$j=1,\dots,N$, where each $A_j$ is a $2\times 2$ matrix with real entries. Weinstock has solved
the problem completely for $N=1$ (in fact, for pairs of transverse, maximally totally-real subspaces 
in $\cplx^n\, \forall n\geq 2$). Using a characterization of simultaneous
triangularizability of $2\times 2$ matrices over the reals, given by Florentino, we deduce a sufficient
condition for local polynomial convexity of the union of the above planes at $0\in \CC$. Weinstock's 
theorem for $\CC$ occurs as a special case of our result.
The picture is much clearer when $N=2$. For three totally-real planes, we shall provide 
an open condition for local polynomial convexity of the union. We shall also argue the optimality 
(in an appropriate  sense) of the conditions in this case.
\end{abstract}

\maketitle

\section{Introduction and statement of results}\label{S:intro}
Let $K$ be a compact subset of $\cplx^n$. The polynomially convex hull of
$K$ is defined by $\hull{K}:= \{z\in \cplx^n : |p(z)|\leq \sup_{K}|p|,\, p\in \cplx [z_1, \dots, z_n]
 \}$. $K$ is said to be polynomially convex if $\hull{K}=K$. We say that a closed subset $E$ of 
$\cplx^n$ is locally polynomially convex at $p\in E$ if $E\cap \ba{\ball(p;r)}$ is polynomially 
convex for some $r>0$ (here, $\ball(p;r)$ denotes the open ball in $\cplx^n$ with centre $p$ and radius $r$). 
In general, it is very difficult to determine whether a given 
compact subset of $\cplx^n$, $n>1$, is polynomially convex. Therefore, researchers have usually 
restricted their attention to specific subclasses of geometric objects. 
In this paper we consider the union of finitely many totally-real planes in $\CC$ intersecting at 
$0\in \CC$, with a {\em mild} transversality condition. In this setting we shall discuss the following: 
\begin{itemize}
 \item  A sufficient condition for the union to be locally polynomially convex at $0\in \CC$ that 
generalizes a theorem (Result~\ref{R:weinstock} below) by Weinstock \cite{Wk} in $\CC$.
\smallskip

\item  An {\em open} condition that is sufficient for the union of totally-real planes 
to be locally polynomially convex at $0\in \CC$ when the number of planes is three.
\smallskip

\item {\em Optimality}, in an appropriate sense defined below, of the above open condition 
for three totally-real planes.
\end{itemize}
We shall see a couple of motivations for focusing attention on the above setting. However, 
let us first make a brief survey of known results in this direction and make the above setting 
a bit more formal. 
\smallskip

It is easy to show that if $M$ is a 
totally-real subspace of $\cplx^n$, then any compact subset of $M$
is polynomially convex. Hence, let us now consider $P_0\cup P_1$, where $P_0$ and $P_1$ are 
two transverse totally-real $n$-dimensional subspaces of $\cplx^n$.
Applying a $\cplx$-linear change of coordinate, we can assume that $P_0= \rea^n$. A careful look
at the second subspace under the same change of coordinate gives us $P_1 = (A+i \id)\rea^n$ for some
$A \in \rea^{n \times n}$ (see \cite{Wk} for details). We shall call this form for the pair of totally-real 
subspaces as {\em Weinstock's normal form}. Weinstock \cite{Wk} found a way of giving complete 
characterization for the polynomial convexity of $P_0\cup P_1$ at $0\in \cplx^n$  in terms of 
the eigenvalues of $A$ (see Result~\ref{R:weinstock} in Section~\ref{S:technical}).
\smallskip

 No analogue of Weinstock's theorem is known for more than two totally-real subspaces. Even in $\CC$, the problem
of generalizing Weinstock's characterization does not seem any simpler. The works of Pascal Thomas 
\cite{T1,T2} give us some sense of the difficulties involved. In \cite{T1} Thomas 
gave an example of a one-parameter family of triples $(P_0^{\eps},P_1^{\eps},P_2^{\eps})$ 
of totally-real planes in $\CC$, intersecting 
at $0\in \CC$, showing that polynomial convexity of each pairwise union at the origin does not 
imply the polynomial convexity of the union at the origin (see Results~\ref{R:PTnpcvx} ). In fact, 
he showed that for the above triples $(P_0^{\eps},P_1^{\eps},P_2^{\eps})$, the polynomial hull of 
$(\cup_{j=1}^3P_j^{\eps})\cap \ba{\ball(0;r)}$ contains an open set in $\CC$. The explicit expression, 
given by Thomas, of this one-parameter family of triples will be used to show optimality, i.e., to  
prove the last assertion of Theorem~\ref{T:3planes}. 
On the other hand, Thomas also found in \cite{T2} examples of triples whose union 
is locally polynomially convex at the origin.
\smallskip

In this paper, we will be far less interested in specific examples of polynomial convexity (or the 
failure thereof) of a finite union of totally-real planes passing through $0\in \CC$.  It turns out 
that many of Weinstock's ideas in \cite{Wk} are the ``correct" ones to follow when one considers 
the union of more than two totally-real subspaces containing $0\in \CC$. One of the motivations 
of this paper is to demonstrate the utility of those ideas. Here, we shall focus closely on how 
the algebraic properties of Weinstock's normal form of a
collection of totally-real planes in $\CC$ influence polynomial convexity. This suggests that 
there is a notion of a Weinstock-type normal form for the union of more than two totally-real 
planes containing $0\in \CC$. Consider a finite collection of maximally totally-real subspaces
$P_0, P_1, \dots, P_N$ of $\cplx^n$, satisfying $P_0\cap P_j=\{0\}$, $j=1,\dots,N$. By exactly the same 
arguments as in \cite{Wk}, we can find a $\cplx$-linear change of coordinate relative to which:  
\begin{align}
  P_0 &: \rea^n \notag\\
 P_j &: M(A_j)=(A_j+i \id)\rea^n, \quad j=1, \dots, N, \label{E:Wnormalform}
\end{align}
where $A_j\in \rea^{n\times n}$, $j=1,\dots, N$. (Note that in this paper {\em we shall refer to 
a $\cplx$-linear operator and its matrix representation relative to the standard basis of $\cplx^n$ 
interchangably}.) We shall call \eqref{E:Wnormalform} {\em Weinstock's normal form for} 
$\{P_0,P_1, \dots, P_N\}$. Note that the collection $\{P_0,P_1, \dots, P_N\}$ above {\em need 
not be pairwise transverse at the origin}. 
\smallskip

When $n=2$, one quickly intuits that Weinstock's methods would work if the $A_j$'s can be 
simultaneously conjugated over $\rea$ to certain cannonical forms. Hence if $\{A_1, \dots, A_N\}$ 
is pairwise commutative then conclusions analogous to Weinstock's can be demonstrated. 
This idea is the basis of our first theorem --- except that the ``commutation condition'' 
above can be weakened.
Note also, that:
 \begin{itemize}
\item We recover Weinstock's theorem for $\CC$ when we take $N=1$ below.
\item We do {\em not} require the planes $P_0, \dots, P_N$ to be pairwise transverse at $0 \in \CC$.
 \end{itemize}
\smallskip

In order to state our first theorem we need the following definition.
\begin{definition}\label{D:reduced}
 A matrix sequence $\mathscr{A}=\{A_1,A_2, \dots, A_N\}$, $A_j \in \rea^{n \times n}$, is
said to be {\em reduced} if there are no commuting pairs among its terms, i.e. $A_jA_k-A_kA_j \neq 0$
for all $1 \leq j<k \leq N$.
 A subsequence $\mathscr{B} \subset \mathscr{A}$ is called a {\em reduction} of $\mathscr{A}$ 
if $\mathscr{B}$ is reduced
and is obtained from $\mathscr{A}$ by deleting some of its terms.
 The {\em reduced length} of $\mathscr{A}$ is the greatest $k \in \mathbb{Z}_+$
such that there exists a reduction $\mathscr{B}\subseteq \mathscr{A}$ of cardinality $k$.
\end{definition}

\begin{theorem}\footnote{This theorem is one of the results of the author's doctoral thesis.} 
\label{T:totrlNplane}
 Let  $P_0,\dots, P_N $ be distinct totally-real planes in $\CC$ containing the origin. Assume
\begin{itemize}
\item[(1)] $P_0 \cap P_j=\{(0,0)\}$ for all $j=1,2,\dots ,N$.
\end{itemize}
Hence, let Weinstock's normal form for $\{P_0,\dots, P_N\}$ be
\begin{align}
 P_0 &: \RR \notag\\
 P_j &: M(A_j)=(A_j+i \id)\RR, \quad j=1,\dots ,N, \notag
\end{align}
where $A_j \in \rea^{2\times 2}$.
\smallskip
Let $L$ be the reduced length of $\{A_1, \dots, A_N\}$. Assume further that for some maximal
reduced subset $\mathscr{B}\subset \{A_1, \dots ,A_N\}$,
\begin{itemize}
\item[(2)] $det[A_j,A_k]=0, \; j\neq k,\ 1\leq j,k \leq N$, and, additionally, $Tr(ABC-CBA)=0 \; \forall
A,B,C \in \mathscr{B}$ if $L=3$.
\end{itemize}
Under these conditions:
\begin{itemize}
\item[(a)]
If each $A_j$ has only real eigenvalues, then $\cup^N_{j=0}P_j$ is polynomially
convex at the origin.
\item[(b)]
If there exists a $j,\ 1\leq j \leq N$, such that $A_j$ has non-real eigenvalues,
then the spectrum of $A_k$ is of the form $\{\lm_k, \ba{\lm_k}\} \; \forall k=1, \dots, N$.
Write $\lm_k=s_k+ it_k$, $V_k=(s^2_k+t^2_k-1, 2s_k), \; k=1,\dots, N$, and $V_0=(1,0)$.
Then, $\cup^N_{j=0}P_j$ is locally polynomially convex at the origin if and only if there
exists no pair $(l,m), \;l\neq m, 0 \leq l, m \leq N$, satisfying $V_l=cV_m$ for some constant $c>0$.
\end{itemize}
\end{theorem}

\noindent Our proof of Theorem~\ref{T:totrlNplane} is strongly influenced by the methods in 
Weinstock's 
paper~\cite{Wk} --- our result is already stated in terms of Weinstock's normal form. However, 
in order to use these techniques (presented in Section~\ref{S:technical}), 
we would like the matrices $A_j, \; j=1,\dots, N$, in \eqref{E:Wnormalform} be as 
simple in structure as possible. Thanks to Lemma~\ref{L:invertible}
 it suffices to work with the planes $P_0, \wtil{P_1}, 
\dots, \wtil{P_N}$, where 
\begin{align}
 \wtil{P_j}\; :\; M(B_j) \, = & \ (B_j+ i\id)\RR, \;\, j=1, \dots, N, \notag \\
 B_j \, \sim & \ A_j \; \text{such that each $B_j$ is sufficiently sparse/structured.} \notag
\end{align}
 It turns out that in the difficult half of Theorem~\ref{T:totrlNplane}, the matrices 
$B_j, \quad j=1, \dots, N$, just need to be upper-triangular. So, in order to exploit Lemma~\ref{L:invertible}:
\begin{itemize}
 \item We require that $\{A_1, \dots, A_N\}$ be {\em 
simultaneously} 
triangularizable. 
\item We also require the common conjugating matrix to be a matrix with {\em real} entries.
\end{itemize}
What is required, hence, is a sufficient condition for simultaneous triangularizability of 
{\em real} $2\times 2$ matrices by a single conjugator with real entries.
This requirement is met by a result of Florentino (see Result~\ref{R:florentino}) and that is where 
Condition $(2)$ comes from. We acknowledge that, under Condition ($2$), the collection 
$\{P_0,P_1,\dots,P_N\}$ is non-generic in the space of $(N+1)$-tuples of totally-real 2-subspaces of 
$\CC$. However, even under this closed condition, Theorem~\ref{T:totrlNplane} has some utility. There 
is a close connection between the polynomial convexity at $0\in \CC$ of a union of $N$ totally-real 
2-subspaces of $\CC$ and the polynomial convexity of the graphs of homogeneous polynomials (in $x$ 
and $y$, $z=x+iy$) of degree $N$, $N\geq 2$. This connection was first investigated by 
Forstneric-Stout \cite{FS} and Thomas \cite{T1} (also see Section~\ref{S:technical} below). Bharali's 
results in \cite{GB1} require the study of the union {\em of more than} two totally-real 2-subspaces, 
and a careful reading reveals that an attempt to weaken his hypothesis is obstructed by a {\em 
non-generic} family of $N$-tuples, $N\geq 3$, of totally-real 2-subspaces in $\CC$. It is hoped 
that Theorem~\ref{T:totrlNplane} will provide some insight into this problem.  
\smallskip

The natural question one may ask is: {\em what happens when Condition $2$ fails?} We investigate 
the situation when the number of totally-real planes is restricted to three. The complexity of the 
set-up with three totally-real planes is low enough that we can replace Condition $(2)$ by 
an open condition. Our next theorem uses essentially same idea as in Theorem~\ref{T:totrlNplane}:

\begin{theorem}\label{T:3planesevrl}
 Let $P_0,P_1, P_2$ be three totally-real planes containing $0\in \CC$.
Assume $P_0 \cap P_j=\{(0,0)\}$ for $j=1,2$.
Hence, let Weinstock's normal form for $\{P_0,P_1, P_2\}$ be
\begin{align}
 P_0 &: \RR \notag\\
 P_j &: M(A_j)=(A_j+i \id)\RR, \quad j=1,2, \notag
\end{align}
where $A_j \in \rea^{2\times 2}$.
Let the pairwise unions of $P_0,P_1,P_2$ be locally
polynomially convex at $0\in \CC$.  
 Then $P_0\cup P_1\cup P_2$ is locally polynomially convex at $0\in \CC$ if any one of the following conditions holds:
\begin{itemize}
 \item[$(i)$]$det[A_1,A_2]>0$ and $detA_j>0, \; j=1,2$,
\smallskip

\item[$(ii)$]$det[A_1,A_2]<0$ and $detA_j<0, \; j=1,2$.
\end{itemize}

\end{theorem}

\noindent It turns out that the first part of Theorem~\ref{T:3planesevrl}, i.e., {\em the case when Condition $(i)$ holds}, 
is a special case of our third 
result. However, we choose to present it separately because of the simplicity of its hypotheses, 
and because this hypothesis provides the motivation for the much more 
technical-looking hypothesis of Theorem~\ref{T:3planes} below.
\smallskip
 
It is easy to see that if $(P_0,P_1,P_2)$ is a triple of totally-real planes with $P_0\cap P_j=\{0\},\; j=1,2$ 
(with one of the three being designated as $P_0$ in case all three planes are mutually transverse), then the matrices 
$A_1$ and $A_2$ associated to Weinstock's normal form for the triple $(P_0,P_1,P_2)$ is unique. In short, 
every triple $(P_0,P_1,P_2)$ of totally-real planes with 
$P_0\cap P_j=\{0\}$, $j=1,2,$ is parametrized by a pair of matrices.
Let us define 
\[
 \OM:= \left\lbrace(A_1,A_2)\in (\rea^{2\times 2})^2 ~:~ \sm(A_1)\subset \cplx \setminus \rea \; 
\text{or}\; (\sm(A_1) \subset \rea \; \text{and}\; \#\sm(A_1)=2) \right\rbrace.
\]
It is clear that $(\rea^{2\times 2})^2 \setminus \OM$ has Lebesgue measure zero. 
(In contrast, it turns out -- see Section~\ref{Proof:3planesevrl} -- that the hypotheses of 
Theorem~\ref{T:3planesevrl} rule out the possibility of $\sm(A_j)\subset \cplx \setminus \rea$, 
$j=1,2$.) In the 
following theorem we will study the triples of totally-real planes parametrized by $\OM$. 

\begin{theorem}\label{T:3planes}
  Let $P_0,P_1, P_2$ be three totally-real planes containing $0\in \CC$.
Assume $P_0 \cap P_j=\{(0,0)\}$ for $j=1,2$.
Hence, let Weinstock's normal form for $\{P_0,P_1, P_2\}$ be
\begin{align}
 P_0 &: \RR \notag\\
 P_j &: M(A_j)=(A_j+i \id)\RR, \quad j=1,2, \notag
\end{align}
and assume $(A_1,A_2)$ belongs to parameter domain $\OM$. 
By definition, $\exists T\in GL(2,\rea)$ such that

\begin{equation}\label{E:normformev}
TA_1T^{-1} = \begin{pmatrix}
              \lm_1 & 0 \\
              	0 & \lm_2
              	\end{pmatrix}
 \; \text{or}\; 
\begin{pmatrix}
s & -t \\
t & s
\end{pmatrix},
\; \lm_1\neq \lm_2 \in \rea, s \in \rea, t\in \rea\setminus \{0\}.
\end{equation}

\noindent Let $\mathscr{A}_1:= TA_1T^{-1}$. Assume further that the pairwise unions of $P_0,P_1,P_2$ be 
locally polynomially convex at $0\in \CC$.
\begin{itemize}
\item[$(i)$]Suppose $\sm(A_1)\subset \rea$ and $\# \sm(A_1)=2$. Then either $det[A_1,A_2]=0$ or 
$\exists T\in GL(2,\rea)$ that satisfies \eqref{E:normformev} and such that $TA_2T^{-1}$ has the form 
\[
 TA_2T^{-1}=\begin{pmatrix}
             s_{12} & t_2 \\
              t_2 & s_{22}
            \end{pmatrix}
\; \text{or}\; 
\begin{pmatrix}
 s_{12} & -t_2 \\
t_2 & s_{22}
\end{pmatrix}
=:\scr{A}_2(T), \; s_{21},s_{22},t_2 \in \rea.
\]
In the first case $P_0\cup P_1\cup P_2$ is locally polynomially convex at $0\in \cplx^2$. In the 
second case $P_0\cup P_1\cup P_2$ is locally polynomially convex if for some $T\in GL(2,\rea)$ as 
stated, $det(\scr{A}_j(T)+\scr{A}_j(T)^{\sf T})>0$, for $j=1,2$.
\medskip

\item[$(ii)$] Suppose $\sm(A_1)\subset \cplx \setminus \rea$. Then $\exists T\in GL(2,\rea)$ that satisfies 
\eqref{E:normformev} and such that 
\[
 TA_2T^{-1}= \begin{pmatrix}
              s_{12} & -t_2\\
               t_2 & s_{22}
             \end{pmatrix}
=: \scr{A}_2(T),\; s_{12},s_{22},t_2 \in \rea.
\]
If $det(\mathscr{A}_j(T)+\mathscr{A}_j(T)^{\sf T})>0$,
$j=1,2$, for some $T\in GL(2,\rea)$ as just stated, then $P_0\cup P_1\cup P_2$ is locally 
polynomially convex at $0\in \CC$.
\end{itemize}
\smallskip

The above conditions are optimal in the sense that, writing 
$\OM^*\varsubsetneq \OM$ to be set of pairs $(A_1,A_2)\in \OM$ that satisfy the conditions in $(i)$
or $(ii)$, there is a one-parameter family of triples $(P_0^{\eps},P_1^{\eps},P_2^{\eps})$  
 parametrized by $(A_1^{\eps}, A_2^{\eps}) \in \OM\setminus \OM^*$ such that
\begin{itemize}
 \item  pairwise unions of $P_0^\eps, P_1^\eps, P_2^\eps$ are locally polynomially convex at the origin; 
\item  the union of the above planes is not locally polynomially convex at $0\in \CC$; and
\item  $(A_1^\eps,A_2^\eps)\tends \bdy \OM^*$ (considered as a subset of $\OM$) as $\eps \searrow 0.$
\end{itemize}
\end{theorem}
 
A few words about the layout of this paper. The first half of the next section collects some 
useful technical results and in the second half we state and prove some useful lemmas in 
linear algebra. In the next three sections (Sections~\ref{Proof:totrlNplane}--\ref{Proof:3planesevrl}), 
we shall give the proof of the theorems. We would like the reader to realise that Part $(i)$ of 
Theorem~\ref{T:3planesevrl} is subsumed by Theorem~\ref{T:3planes}. For this reason, we shall  
prove Theorem~\ref{T:3planesevrl} in Section~\ref{Proof:3planesevrl} after we prove Theorem~\ref{T:3planes}.

\section{Technical preliminaries}\label{S:technical}

We shall require some preliminaries to set the stage for proving the theorems. First, we 
state a lemma from Weinstock's paper \cite{Wk} --- whose proof is quite easy --- that allows 
us to conjugate the matrices coming from Weinstock's normal form by real nonsingular matrices. 

\begin{lemma}\label{L:invertible}
 Let $T$ be a invertible linear operator on $\cplx^n$ whose matrix representation with respect 
to the standard basis is an $n \times n$ matrix with real entries. Then $T$ maps 
$M(A)\cup \rea^n$ onto $M(TAT^{-1}) \cup \rea^n$.
\end{lemma}

We now state the result by Weinstock \cite{Wk} which was already referred repeatedly in Section~\ref{S:intro}. 
This theorem will play a vital role in the proof of Theorem~\ref{T:totrlNplane}.
\begin{result}[Weinstock] \label{R:weinstock}
 Suppose $P_1$ and $P_2$ are two totally-real subspaces of $\cplx^n$ of maximal dimension intersecting
only at $0 \in \cplx^n$. Denote the normal form for this pair as:
\begin{align}
 P_1 &:  \rea^n, \notag\\
P_2 &: (A+i \id)\rea^n. \notag
\end{align}
$P_1 \cup P_2$ is locally polynomially convex at the origin if and only if $A$ has no purely
imaginary eigenvalue of modulus greater than 1.
\end{result}

Next, we state  two lemmas from the literature which will be used repeatedly in the proofs of 
our theorems. 
The first one --- due to Kallin \cite{K} --- deals with the polynomial convexity of the union of 
two polynomially convex sets. The second one --- which is a version of a lemma from 
Stolzenberg's paper \cite[Lemma 5]{St1} (also see Stout's book \cite{S1} ) --- gives a criterion for 
polynomial convexity of a compact set $K$ in terms of the existence of a function that 
belongs to the uniform algebra on $K$ generated by the polynomials and satisfies 
some special property.

\begin{lemma}[Kallin]\label{L:kallin}
 Let $K_1$ and $K_2$ be two compact polynomially convex subsets in $\cplx ^n$. Suppose
$L_1$ and $L_2$ are two compact polynomially convex subsets of $\cplx$ with 
$L_1 \cap L_2 =\{0\}$. Suppose further that there exists a holomorphic polynomial 
$P$ satisfying the following conditions:
 \begin{enumerate}
 \item[$(i)$] $P(K_1) \subset L_1$ and  $P(K_2) \subset L_2$; and
 \item[$(ii)$] $P^{-1}\{0\} \cap (K_1 \cup K_2)$ is polynomially convex.
 \end{enumerate}
 Then $K_1 \cup K_2$ is polynomially convex.
\end{lemma}

Given a compact $X\subset \cplx^n$, $\poly(X)$ will denote the uniform algebra on $X$ generated by holomorphic 
polynomials.
\begin{lemma}[Stolzenberg] \label{L:S1}
Let $X \subset \cplx^n$ be compact. Assume $\poly(X)$ contains a function $f$ such that $f(X)$ has empty interior 
and $\cplx \setminus f(X)$ is connected. Then, $X$ is polynomially convex if and only if $f^{-1}\{w\} \cap X$ 
is polynomially convex for each $w \in f(X)$ .
 \end{lemma}

\noindent The version of Lemma~\ref{L:S1} that we stated above originates in a remark 
following~\cite[Theorem 1.2.16]{S1} in Stout's book.  
\smallskip

The next result, due to Florentino \cite{Fl}, concerns the simultaneous triangularizability of a family 
of matrices over the field of real numbers. Though the following theorem is 
actually valid over any integral domain, we shall state it over the field of real numbers, which 
is the field relevant to our proof.  

\begin{result}[Florentino, \cite{Fl}]\label{R:florentino}
 Let $\mathscr{A}=(A_1,\dots,A_n),\;A_j \in \rea^{2 \times 2}$, have reduced length $l\leq n$.
 Let $\mathscr{A'}\subset \mathscr{A} $ be a maximal reduction and, without loss of generality, 
let $\scr{A'}= (A_1, \dots, A_l)$. Then:

\begin{itemize}
\item[$(i)$] If $l=3$, $\mathscr{A}$ is triangularizable if and only if each $A_k$ is triangularizable, 
$det[A_j,A_k]=0,\;j,k\leq l$, and $Tr(ABC-CBA)=0$ for all $A,B,C \in \mathscr{A'}$. 
\item[$(ii)$] If $l \neq 3$, $\mathscr{A}$ is triangularizable if and only if each $A_k$ is 
triangularizable and $det[A_j,A_k]=0,\;j,k\leq l$.
\end{itemize}
\end{result}
\noindent We must clarify that, in the above result, the expression ``$A_k$ is triangularizable'' means 
that $A_k$ is similar to a real upper triangular matrix by conjugation with a {\em real} invertible matrix. 
Likewise, the expression ``$\mathscr{A}$ is triangularizable'' means that each member of $\mathscr{A}$ 
is triangularizable by conjugation by the same matrix. We refer to the reader to Definintion~\ref{D:reduced} 
for the definitions of reduction and reduced length. 
\smallskip

We can now appreciate the complexity of Condition $(2)$ in Theorem~\ref{T:totrlNplane}; the latter 
half of this condition is inherited from part $(i)$ of Result~\ref{R:florentino}. The case $l=3$ 
is genuinely exceptional. Florentino shows in \cite[Example~2.11]{Fl} that the condition 
$tr(A_1A_2A_3-A_3A_2A_1)=0$ cannot, in general, be dropped.
\medskip

Let us now state a result by Thomas~\cite{T1} which will play the key role in our argument in the proof of the 
optimality-part of Theorem~\ref{T:3planes}.
\begin{result}[Thomas, \cite{T1}]\label{R:PTnpcvx}
 There exist three pairwise transversal totally-real planes $P_j,\; 1 \leq j \leq 3$, in $\CC$
passing through origin such that:
\begin{itemize}
\item[$(i)$] $P_j \cup P_k$ is locally polynomially convex at $0 \in \CC$ for all $j \neq k $;
\item[$(ii)$] $((P_1 \cup P_2 \cup P_3) \cap \ba{\ball(0;1)} \; \hull{)}$ contains an open ball in $\CC$.
\end{itemize}
\end{result}

\noindent Note that, in the statement $(ii)$ of the above theorem, the radius of the closed
ball has no significant role. Since the set $P_1 \cup P_2 \cup P_3$ is invariant under all
real dilations, $(ii)$ would hold true with any $\ball(0;r)$, $r>0$, replacing the unit ball. 
We will see some more discussions on these planes \cite{T1} in 
Section~\ref{Proof:3planes}.
\smallskip

We now prove some lemmas that will be used in the proofs of the theorems. All the lemmas are 
linear algebraic in nature. We also prove a proposition --- an identity showing conditions of 
Theorem~\ref{T:3planes} are invariant under conjugation --- at the end of this section.
 \begin{lemma}\label{L:jordontype}
  Let $A \in \rea^{2\times 2}$ and suppose $A$ has non-real eigenvalues $p\pm iq$.
 Then, there exists $S \in GL(2,\rea)$ such that
\[
 S^{-1}AS=   \begin{pmatrix}
 p & -q \\
 q & p
\end{pmatrix}.
\]
\end{lemma}
\begin{proof}
Let $v$ be an eigenvector of $A$ corresponding to the eigenvalue
$p +iq$. Hence, $\ba{v}$ is an eigenvector of $A$ corresponding to the eigenvalue $p-iq$.
 Since $q \neq 0$, the set $\{v, \ba{v}\}$ is linearly independent over $\cplx$. Now,
writing $v=v_1+iv_2$, where $v_1,v_2 \in \RR$, we have the following:
\begin{equation}\label{E:Av}
Av=(p+iq)v \Longrightarrow Av_1= p v_1- q v_2 ~~ \text{and}~~ Av_2= q v_1+ p v_2.
\end{equation}
 Since $\{v, \ba{v}\}$ is linearly independent over $\cplx$, $\{v_1, v_2\}$ is also linearly
independent over $\cplx$. By \eqref{E:Av}, the representation of $A$ with respect to the basis $\{v_1, v_2\}$ is
$ \begin{pmatrix}
 p & -q \\
 q &  p
\end{pmatrix} $.
Hence, the basis-change matrix that transform $A$ to its representation with respect to the basis
$\{v_1, v_2\}$ is the desired real invertible matrix $S$.
\end{proof}

\begin{lemma}\label{L:rcommutator10}
Let $A_1,....., A_N \in \rea^{2\times 2}$ and assume that $det[A_j,A_k]=0\; \forall j\neq k$.
Assume also that $A_1$ has non-real eigenvalues.Then, there exists $S \in GL(2,\rea)$ such that
\[ S^{-1}A_jS=
 \begin{pmatrix}
 s_j & -t_j \\
 t_j&  s_j
\end{pmatrix},\quad j=1,\dots ,N,
\]
where $s_j,t_j \in \rea$, $j=1,\dots, N.$
 \end{lemma}
\begin{proof}
  Since $A_1$ has non-real eigenvalues, by Lemma \ref{L:jordontype}, there exists $S \in GL(2,\rea)$ such that
\[
 S^{-1}A_1S=   \begin{pmatrix}
 p & -q \\
 q & p
\end{pmatrix},
\]
where $p \pm iq \,(q\neq 0)$ are the eigenvalues of $A_1$. Suppose
\[
 S^{-1}A_jS=
 \begin{pmatrix}
 a & b \\
 c & d
\end{pmatrix}, \; \text{for some $j: 1 \leq j \leq N$}.
\]
Since the determinant remains invariant under conjugation by an invertible matrix, by hypothesis:
\begin{equation}\label{E:det0}
det[S^{-1}A_1S,S^{-1}A_jS]=0.
\end{equation}
A simple calculation gives us
\[
 [S^{-1}A_1S,S^{-1}A_jS] =
\begin{pmatrix}
 -q(b+c) & q(a-d) \\
 q(a-d) & q(b+c)
\end{pmatrix}.
\]
Hence, from \eqref{E:det0} and the fact that $q \neq 0$, we infer that $b=-c, \; a=d$.
Thus, under the conjugation by $S$, we have
\[ S^{-1}A_jS=
 \begin{pmatrix}
 s_j & -t_j \\
 t_j&  s_j
\end{pmatrix},
\]
where $s_j,t_j \in \rea$, $j=1,\dots, N.$
\end{proof}

\begin{lemma}\label{a1a2}
 Let $A_1, A_2 \in \rea^{2 \times 2}$ be two matrices such that $det[A_1,A_2]=0$ and $A_1-A_2$ 
is invertible. Suppose $A_1$ has non-real eigenvalues. Then
\begin{itemize}
\item $A_2$ either has non-real eigenvalues or is a scalar matrix; and
\item
 $B:= (A_1A_2+ \id)(A_1-A_2)^{-1}$ has complex conjugate eigenvalues.
\end{itemize}
\end{lemma}

\begin{proof}
 Since $A_1$ has non-real eigenvalues and $det[A_1,A_2]=0$, appealing Lemma \ref{L:rcommutator10}, we see that
there exists $S \in GL(2,\rea)$ such that
\[ S^{-1}A_jS=
 \begin{pmatrix}
 s_j & -t_j \\
 t_j&  s_j
\end{pmatrix},~~j=1,2.
\]
This shows that $A_2$ either has non-real eigenvalues or is a scalar matrix.
\smallskip
Note that conjugating by the matrix $\begin{pmatrix}
  1 & i \\
  i & 1
\end{pmatrix}$, we see that
\[ A_j \sim
 \begin{pmatrix}
 \ba{\lm_j} & 0 \\
  0 &  \lm_j
\end{pmatrix}, \;j=1,2.
\]
Hence, $(A_1-A_2)$ and $(A_1A_2+I)$ can be conjugated by $S$ to diagonal matrices with diagonal entries
$(\ba{\lm_1 - \lm_2}, \lm_1-\lm_2)$ and $(\ba{\lm_1 \lm_2}+1, \lm_1 \lm_2 +1)$ respectively. Hence, by
examining $S^{-1}BS$, we see that the matrix $B=(A_1A_2+I)(A_1-A_2)^{-1}$ has complex 
conjugate eigenvalues.
\end{proof}

\begin{lemma}\label{L:norform3planes}
 Let $A_1,A_2 \in \rea^{2\times 2}$. Suppose $A_1$ has two distinct eigenvalues. Then
$\exists T\in GL(2,\rea)$ such that:
\begin{itemize}
 \item[$(i)$]  If $A_1$ has real eigenvalues and $det[A_1,A_2]\neq 0$, then
\[
TA_1T^{-1}=\begin{pmatrix}
                       \lm_1 & 0 \\
                       0 & \lm_2
                      \end{pmatrix}\;\; \text{and} \;\;
  TA_2T^{-1}= \begin{pmatrix}
                       s_{21} & t_2 \\
                       t_2 & s_{22}
                      \end{pmatrix}
\;\; \text{or} \;\;   \begin{pmatrix}
                       s_{21} & -t_2 \\
                       t_2 & s_{22}
                      \end{pmatrix} 
\]
 for  $\lm_j,s_{2j}, t_2 \in \rea, \; j=1,2$, 
\smallskip

\item[$(ii)$] If $A_1$ has non-real eigenvalues, then
\[TA_1T^{-1}= \begin{pmatrix}
                       s_{1} & -t_1\\
                       t_1 & s_{1}
                      \end{pmatrix}
 \;\; \text{and}\;\;  TA_2T^{-1} = \begin{pmatrix}
                       s_{21} & -t_2 \\
                       t_2 & s_{22}
                      \end{pmatrix} 
\]
 for $s_j,s_{2j}, t_j \in \rea, \; j=1,2$. 
\end{itemize}
\end{lemma}

\begin{proof}
\noindent $(i)$ Since $A_1$ has two distinct real eigenvalues, 
$A_1$ is diagonalizable over $\rea$, i.e. there exists a $S\in GL(2,\rea)$ such
that
\[
 SA_1S^{-1}=\begin{pmatrix}
                       \lm_1 & 0 \\
                       0 & \lm_2
                      \end{pmatrix},
\;\lm_1\neq \lm_2\in \rea.\]
Hence, without loss of generality, we can assume that $A_1= \begin{pmatrix}
                                                          \lm_1 & 0 \\
                                                            0 & \lm_2
                                                           \end{pmatrix}$.
Suppose $A_2 = \begin{pmatrix}
                       s_{21} & t_1 \\
                       t_2 & s_{22}
                      \end{pmatrix}$, 
$t_j, s_{2j}\in \rea$ for $j=1,2$. Observe that, in view of Result~\ref{R:florentino}, 
$t_1t_2=0 \Leftrightarrow det[A_1,A_2]=0$. Hence neither $t_1$ nor $t_2$ is zero.
We have, since $A_1$ commutes with all diagonal matrices, that 
\[
GA_1G^{-1}= A_1 \; \text{for}\; 
G= \begin{pmatrix}
                       g_1 & 0 \\
                       0 & g_2
                      \end{pmatrix},\;
\text{where}\; g_1g_2 \neq 0.
\]
 We also have, after conjugating $A_2$ by $G$, that
\begin{equation}\label{E:normforma2t1t2samesign}
 GA_2G^{-1}= \begin{pmatrix}
                       s_{21} & t_1g_1/g_2 \\
                       t_2g_2/g_1 & s_{22}
                      \end{pmatrix}.
\end{equation}

\noindent Now observe that if $t_1$ and $t_2$ are of same sign, then there exist $g_1, g_2 \in
\rea \setminus \{0\}$ such that
\[
 t_1g_1^2=t_2g_2^2.
\]
Therefore, in this case,
$\wtil{t_2}:= t_1g_1/g_2=t_2g_2/g_1$, and we conclude from \eqref{E:normforma2t1t2samesign} that 
\[
GA_2G^{-1}= \begin{pmatrix}
                       s_{21} & \wtil{t_2} \\
                      \wtil{t_2} & s_{22}
                      \end{pmatrix}.
\]
\smallskip

\noindent We also observe that, if $t_1$ and $t_2$ are of different sign, then
there exist $g_1, g_2 \in \rea\setminus \{0\}$ such that
\[
 t_1g_1^2+t_2g_2^2=0.
\]
Therefore, in this case, 
$\wtil{t_2}:= -t_1g_1/g_2=t_2g_2/g_1$, and we conclude from \eqref{E:normforma2t1t2samesign} that 
\[
GA_2G^{-1}= \begin{pmatrix}
                       s_{21} & -\wtil{t_2} \\
                       \wtil{t_2} & s_{22}
                      \end{pmatrix}.
\]
\bigskip

 \noindent $(ii)$ Since $A_1$ has non-real eigenvalues, say 
$s_1\pm it_1$, by Lemma~\ref{L:jordontype} there exists $S\in GL(2,\rea)$ such 
that
$SA_1S^{-1}= \begin{pmatrix}
                       s_{1} & -t_1\\
                       t_1 & s_{1}
                      \end{pmatrix}$, $s_1,t_1\in \rea $.
Hence, without loss of generality, we may assume that  
$A_1= \begin{pmatrix}
                       s_{1} & -t_1\\
                       t_1 & s_{1}
                      \end{pmatrix}$. 
Let 
\[
A_2= \begin{pmatrix}
                       m_{1} & m_2\\
                       m_3 & m_{4}
                      \end{pmatrix}, \;m_j\in \rea,\; j=1,2,3,4, 
\] 
with $m_2+m_3 \neq 0$; otherwise, there is nothing to prove.                                            
\smallskip
                      
We observe that $A_1$ commutes with all the matrices having the same structure 
as that of itself. Let $G:= \begin{pmatrix}
                       g_{1} & -g_2\\
                       g_2 & g_{1}
                      \end{pmatrix}$ 
with $g_1,g_2\in \rea$, $g_1^2+g_2^2=1$ and $g_1g_2\neq 0$.  Therefore, 
\[
GA_1G^{-1}= A_1,
\]             
and 
\begin{align}
GA_2G^{-1}&= \begin{pmatrix}
                       g_{1} & -g_2\\
                       g_2 & g_{1}
                      \end{pmatrix}
                      \begin{pmatrix}
                       m_{1} & m_2\\
                       m_3 & m_{4}
                      \end{pmatrix}
                      \begin{pmatrix}
                       g_{1} & g_2\\
                       -g_2 & g_{1}
                      \end{pmatrix} \notag\\
            &= \begin{pmatrix}
               g_1^2m_1-g_1g_2(m_2+m_3)+ g_2^2m_4 & g_1^2m_2+ g_1g_2(m_1-m_4)-g_2^2m_3 \\
               g_1^2m_3+ g_1g_2(m_1-m_4)-g_2^2m_2 &  g_2^2m_1+g_1g_2(m_2+m_3)+ g_1^2m_4 
               \end{pmatrix} \notag \\
            &=: \begin{pmatrix}
               f_1(g_1,g_2) & f_2(g_1,g_2) \\
               f_3(g_1,g_2) & f_4(g_1,g_2)
               \end{pmatrix}.   \notag        
\end{align}
\smallskip

Let us now look closely at the quadratic equation $f_2(g_1,g_2)+f_3(g_1,g_2)=0$ 
in $g_1,g_2$. This gives: 
\[
(m_2+m_3)(g_1^2-g_2^2)+2(m_1-m_4)g_1g_2=0 
\]
This implies, since $m_2+m_3 \neq 0$ and $g_1g_2 \neq 0$, 
\[
 \frac{g_1}{g_2} - \frac{g_2}{g_1} + 2\frac{m_1-m_4}{m_2+m_3} =0
\]
Looking the above as a quadratic in $\frac{g_1}{g_2}=:\mu$, we have 
\begin{equation}\label{E:normformquadratic}
 \mu^2+2 \frac{m_1-m_4}{m_2+m_3} \mu -1 =0. 
\end{equation}
The discriminant of the above quadratic is 
\[
 4 \left(\frac{m_1-m_4}{m_2+m_3} \right)^2 + 4, 
\]
 which is greater than zero for all $m_j\in \rea$, $j=1,2,3,4$. Hence \eqref{E:normformquadratic} has a real 
root, say $\mu_1$. 
Therefore, taking $G = \begin{pmatrix}
                            \mu_1 & -1\\
                             1 & \mu_1
                           \end{pmatrix}$, 
we have 
\begin{align}
 GA_1G^{-1} &= \begin{pmatrix}
              s_1 & -t_1 \\
              t_1 & s_1
             \end{pmatrix} \; \text{and}\;
GA_2G^{-1} = \begin{pmatrix}
               s_{21} & -t_2 \\
               t_2 & s_{22}
              \end{pmatrix}. \notag
\end{align}
\end{proof}

\section{The proof of Theorem~\ref{T:totrlNplane}} \label{Proof:totrlNplane}

We remind the reader that, in this section, the word ``triangularizable'' will refer to 
triangularization by a {\em real} matrix. 
\smallskip

\begin{proof}[Proof of Theorem~\ref{T:totrlNplane}]
 \noindent (a) Since all $A_j,\; j=1, \dots, N$, have real eigenvalues, each $A_j,\; j=1, \dots,N$, is triangularizable.
 We now appeal to Result~\ref{R:florentino} to get:
\[
 A_j \sim \begin{pmatrix}
 \mu_j & a_j\\
 0 & \nu_j
\end{pmatrix} ,~~~\mu_j,\nu_j, a_j \in \rea,\, j=1,...,N,
\]
by the conjugation with a common matrix $S\in GL(2, \rea)$. In view
of the Lemma \ref{L:invertible}, we may assume that
\[
 A_j = \begin{pmatrix}
 \mu_j & a_j\\
 0 & \nu_j
\end{pmatrix} ,\; \mu_j,\nu_j, a_j \in \rea,\, j=1,...,N.
 \]
>From this point, we will --- for simplicity of notation --- refer to each $S(P_j)$ 
as $P_j$, $j=0, \dots, N$. Hence, $P_j=M(A_j)$ for the preceding choice of $A_j$, 
$j=1, \dots,N$. We have
\[
 M(A_j)=\{((\mu_j+i)x+a_jy, (\nu_j+i)y): x,y \in \rea\},~~~\text{for all}~~ j=1,\dots,N.
\]
Let $K:= (\cup^N_{j=0}P_j) \cap \ba{\ball(0,1)}$ and $K_j:= P_j \cap K,\;  j=0, \dots, N $.
Since, for every $j=0,\dots,N$, $K_j$ is a compact subset of a totally real plane, $K_j$ is polynomially
convex. We shall use Lemma~\ref{L:S1} to show the polynomial convexity of $K= \cup_{j=1}^NK_j$. Consider the 
polynomial
\[
 F(z,w)=w .
\]
Clearly, there exists a real number $R>0$ such that:
\begin{equation} \label{E:imageK}
 F(K) \subset (\cup_{j=1}^N \{(\nu_j+i)y : y \in \rea, |y|\leq R\})\cup \{y : y \in \rea, |y| \leq R\}.
\end{equation}
Since each of the members in the union of the right hand side of \eqref{E:imageK} is a bounded real line segment 
in $\cplx$, $F(K)$ has no interior and $\cplx \setminus F(K)$ is connected. We shall now calculate 
$F^{-1}\{\zt\}\cap K$, where $\zt \in F(K)$. If $\zt \neq 0$, then $\zt \in F(K_{j^0})$ for 
some $j^0 \leq N$. Hence, we get
\begin{align}\label{E:non0t}
 F^{-1}\{\zt\}\cap K &=\begin{cases}
                     \{((\mu_{j^0}+i)x+ a_{j^0}\zt/(\nu_{j^0}+i), \zt) : x \in \rea\}\cap K, \;
                                                 & \text{if  $1 \leq j^0\leq N$},\\
                     \{(x,\zt)\in \CC : x \in \rea\} \cap K, \; & \text{if $j^0=0$}.
                      \end{cases}
\end{align}
We remark that $\zt\in F(K_{j^0})$ implies that $\zt/(\nu_{j^0}+i) \in \rea$. Also, $\nu_{j^0}+i \neq 0$ 
because, in the present case, $\nu_{j^0}\in \rea$.
If $\zt=0$, then we have 
\begin{align} \label{E:0t}
  F^{-1}\{0\}\cap K &= \left[ (\cup_{j=1}^N \{((\mu_j+i)x,0) \in \CC : x\in \rea\}) \cup \{(x,0) \in \CC : x \in \rea\} 
\right] \cap K
\end{align}
In view of \eqref{E:non0t}, we have the set $F^{-1}\{\zt\}\cap K$ is a single line segment in $\CC$ when 
$\zt \neq 0$. Hence,
\begin{equation}\label{E:inversenon0}
 ( F^{-1}\{\zt\}\cap K \hull{)} = F^{-1}\{\zt\}\cap K, \; \text{for $\zt \neq 0$}.
\end{equation}
>From \eqref{E:0t}, we see that $F^{-1}\{0\}\cap K$ is a union of line segments in $\cplx_z \times \{0\}$
intersecting only at the origin. Hence,
\begin{equation}\label{E:inverse0}
  ( F^{-1}\{0\}\cap K \hull{)} = F^{-1}\{0\}\cap K.
\end{equation}
Thus, in view of \eqref{E:inversenon0} and \eqref{E:inverse0}, we can appeal to Lemma~\ref{L:S1} to get 
 the polynomial convexity of $K$.
\medskip
  
\noindent (b) Without loss of generality, assume that $A_1$ has non-real complex eigenvalues. From 
Lemma~\ref{L:rcommutator10}, it follows that there exists $S \in GL(2,\rea)$ such that
\[ S^{-1}A_jS=
 \begin{pmatrix}
 s_j & -t_j \\
 t_j&  s_j
\end{pmatrix},\,s_j, t_j \in \rea, \;  j=1,\dots,N.
\]
Again, by Lemma \ref{L:invertible}, there is no loss of generality to assume
\[
A_j= \begin{pmatrix}
 s_j & -t_j \\
 t_j&  s_j
\end{pmatrix}, \,s_j, t_j \in \rea \; j=1,\dots,N.
\]
We will relabel $S(P_j)$ as $P_j$, $j=0, \dots, N$, exactly as in (a). Then
\[
 M(A_j)=\{((s_j+i)x-t_jy, t_jx+(s_j+i)y):~~x,y \in \rea \}, \;j=1,\dots,N.
\]
Note that, in this case, 
\begin{equation}\label{E:detA1A2}
 det(A_j-A_k) = (s_j-s_k)^2 +(t_j-t_k)^2 >0,
\end{equation}
and consequently, $M(A_j) \cap M(A_k)= \{0\}$ for $j \neq k$.

Suppose there does not exist any constant $c>0$ such that
\[
 V_j= c V_k,\; \text{for some  $j \neq k,\ 0 \leq j,k \leq N$}.
\]
 We shall show that $K:= (\cup^N_{j=0}P_j) \cap \ba{\ball(0,1)}= \cup_{j=0}^NK_j$ 
is polynomially convex, where $K_j:= P_j\cap K,\; 0 \leq j \leq N$. Each $K_j$ 
is necessarily polynomially convex. In this case we shall use Kallin's lemma 
(i.e Lemma~\ref{L:kallin}) to show that $K$ is polynomially convex. 
Let us consider the polynomial
\[
 F(z,w):=z^2+w^2.
\]
We shall now look at the image of $K_j$ under this map:
\begin{equation}\label{E:imageK0}
F(K_0) \subset \{z\in \cplx:~z\geq 0 \}= \{\al V_0:~ \al \geq 0\},
\end{equation}

\begin{align*}
 F((s_j+i)x-t_jy, t_jx+(s_j+i)y) &= ((s_j+i)x-t_jy)^2+(t_jx+(s_j+i)y)^2\\
  &= (s_j^2+t_j^2-1)(x^2+y^2)+ 2is_j(x^2+y^2).
\end{align*}
Hence,
\begin{equation}\label{E:imageKj}
 F(K_j) \subset \{\beta V_j \in \RR:~ \beta \geq 0\},\;j=1,\dots,N.
\end{equation}
Since there does not exist any constant $c>0$ such that
 $V_j= cV_k$, for some $ j \neq k,\;0 \leq j,k \leq N$,
we have by equations \eqref{E:imageK0} and \eqref{E:imageKj}:
\begin{equation}\label{E:kallincond1}
 F(K_l)\cap F(K_m)=\{0\}\;\text{for}\ l\neq m.
\end{equation}
Furthermore,
\begin{equation}\label{E:kallincond2}
F^{-1}\{0\}\cap K_j=\{0\} ~~\text{for all}~~ j=0,\dots,N. 
\end{equation}
>From \eqref{E:kallincond1}, \eqref{E:kallincond2},  it follows that:
\begin{itemize}
\item for each $j$, $F(K_j) $ lies in different line segment of $\cplx$, each of which has
an end at $0\in \CC$; and 
\item $F^{-1}\{0\} \cap K= \{0\}$, which is polynomially convex.
\end{itemize}
Since each $F(K_j),\; j=0, \dots, N$, is polynomially convex, the above shows that all the conditions 
of Kallin's lemma are satisfied. Hence, $K= \cup_{j=1}^N K_j$ is polynomially convex; i.e. 
$\cup_{j=0}^N P_j$ is locally polynomially convex at the origin.
\smallskip

We will prove the converse in its contrapositive fomulation.
Let there exist two numbers $l, m$ such that $l \neq m$ and for some constant $c>0$
\begin{equation*}
 V_l=cV_m.
\end{equation*}
This implies 
\begin{equation}\label{E:deteq1}
 det(A_l+i\id)=c.det(A_m+i\id).
\end{equation}
Without loss of generality, let us assume that $l=2$ and $m=1.$
>From \eqref{E:deteq1}, we get
\begin{equation}\label{E:deteq2}
 det[(A_1-i\id)(A_2+i\id)]= c.det(A_1^2+\id) >0.
\end{equation}
Note that if we view $(A_1-i\id)$ as a $\cplx$-linear transformation on $\CC$, then
\begin{align*}
 (A_1-i\id)(M(A_1)) &= (A_1^2+\id)\RR = \RR \\
(A_1-i\id)(M(A_2)) &=[A_1A_2+ \id +i(A_1-A_2)]\RR \\
&=[(A_1A_2+\id)(A_1-A_2)^{-1}+i\id]\RR \equiv (B+i\id)\RR.
\end{align*}
The first equality follows from the fact that $(A_1^2+\id)$ is invertible (because of \eqref{E:deteq2} 
and the fact that $M(A_1)$ is totally-real) and the invertibility of $(A_1 - A_2)$ follows from \eqref{E:detA1A2}.
Now, from Lemma \ref{a1a2}, we get that $B$ has complex conjugate eigenvalues. We now write $\sigma(B)=\{\mu, \ba{\mu}\}$,
and $\mu = s+it,~~s,t \in \rea$. 
>From \eqref{E:detA1A2} and \eqref{E:deteq2}, we get:
\[
 det(B+i\id)=(s^2+t^2-1)+2is >0.
\]
This implies $s=0$ and $|t|>1$. From an auxiliery result of Weinstock \cite[Theorem 2]{Wk}, it
follows that $\RR \cup M(B)=(A_1-i\id)(M(A_1)\cup M(A_2))$ is 
not locally polynomially convex at the origin. Equivalently, $\cup_{j=0}^NP_j$ is not locally 
polynomially convex at the origin.
\end{proof}

Let us make the following remark. 
\begin{remark}
 We observe that under Condition $(2)$ of the above theorem, local polynomial convexity of pairwise unions 
of $P_0, \dots, P_N$ at the origin imples local polynomial convexity of $\cup_{j=0}^nP_j$ at $0\in \CC$.
\end{remark}
\medskip

\section{Proof of the Theorem~\ref{T:3planes}}\label{Proof:3planes}
Before proceeding to the proof of Theorem~\ref{T:3planes}, we shall state some preliminaries 
needed in the proof of optimality part of Theorem~\ref{T:3planes}. Recall Result~\ref{R:PTnpcvx}, 
which gives a triple of totally-real planes whose union is not locally polynomially convex at $0\in\CC$ although 
each of the pairwise unions is locally polynomially convex at the origin. In the proof of Result~\ref{R:PTnpcvx}, 
Thomas \cite{T1} demonstrates a family of triples $(P_0^{\eps},P_1^{\eps},P_2^{\eps})$, where $\eps$ 
is a complex number close to $0$, having the above mentioned property. The planes in the above 
triples are graphs in $\CC$ with the following equations: 
\begin{align}
P_0^{\eps}~&:~ w=\ba{z} \notag\\
P_1^{\eps}~&:~ w=-\frac{\sqrt{3}(\sqrt{3}-i)}{2\eps}z+\frac{-1+\sqrt{3}i}{2}\ba{z} \notag\\
P_2^{\eps}~&:~ w= -\frac{\sqrt{3}(\sqrt{3}+i)}{2\eps}z-\frac{1+\sqrt{3}i}{2}\ba{z}. \notag  
\end{align}
In the proof of Theorem~\ref{T:3planes}, we shall restrict our attention to the above triples when $\eps \in 
\rea \setminus \{0\}$. 
We now apply a $\cplx$-linear change of coordinate $(z,w) \longmapsto (z+w, i(w-z))$ from $\CC$ to $\CC$. 
In the new coodinate, we have 
\begin{align}
P_0^{\eps} ~&:~ \RR \notag\\
P_j^{\eps} ~&:~ (A_j^{\eps}+i\id)\RR, \; j=1,2,  \notag
\end{align}  
where $A_j^{\eps}\in \rea^{2\times 2}$ have the following form: 
\begin{equation}\label{E:PTmatrices}
A_1^{\eps}= \begin{pmatrix}
                    \frac{\eps}{\sqrt{3}(1+\eps)} & \frac{1}{1+\eps} \\
                    -\frac{1}{1-\eps}  & -\frac{\eps}{\sqrt{3}(1-\eps)}
                    \end{pmatrix} 
\; \text{and}\; 
A_2^{\eps}= \begin{pmatrix}
                    -\frac{\eps}{\sqrt{3}(1+\eps)} & \frac{1}{1+\eps} \\
                    -\frac{1}{1-\eps}  & \frac{\eps}{\sqrt{3}(1-\eps)}
                    \end{pmatrix}. 
\end{equation}
\smallskip

We are now in a position to begin the proof of Theorem~\ref{T:3planes}.
\smallskip

\noindent{\em The proof of Theorem~\ref{T:3planes}:}
In view of Lemma~\ref{L:norform3planes}, we divide the proof of this theorem into 
two cases depending on the eigenvalues of $A_1$.
\smallskip

 \noindent {\em {\bf Case I.} When eigenvalues of $A_1$ are real and distinct.}
\smallskip

\noindent First, let us consider the case when $det[A_1,A_2]=0$. By Lemma~\ref{L:rcommutator10}, 
$A_1$ and $A_2$ both have real eigenvalues, whence they are triangularizable over $\rea$. 
Hence, Theorem~\ref{T:totrlNplane} applies, and from Part (a) of that theorem, we are done.
\smallskip

When $det[A_1,A_2]\neq 0$, the first assertion of $(i)$ of our theorem follows from Part $(i)$ of
Lemma~\ref{L:norform3planes}. By hypothesis, $det(\scr{A}_1+\scr{A}_1^{\sf T})>0$, 
and there is a $T_0\in GL(2,\rea)$ such that $det(\scr{A}_2(T_0)+\scr{A}_2(T_0)^{\sf T})>0$.
For simplicity of notation, for the remainder of this proof, we shall write $\scr{A}_2:=\scr{A}_2(T_0)$. 
Hence: 
\[
\mathscr{A}_2:= T_0A_2T_0^{-1} = \begin{pmatrix}
                            s_{21} & t_2 \\
                            t_2 & s_{22}
                           \end{pmatrix} 
\; \text{or}\; \begin{pmatrix}
                s_{21} & -t_2 \\
                t_2 & s_{22}
               \end{pmatrix}.
\]
We shall now divide the proof into two subcases. Once again, for simplicty of notation, {\em we shall 
follow the conventions of the proof of Theorem~\ref{T:totrlNplane} and denote the planes 
$T_0(P_j) as P_j$, $j=0,1,2$}.
\smallskip

 \noindent {\em {\bf (a)} When $\mathscr{A}_1= \begin{pmatrix}
                                             \lm_1 & 0 \\
                                               0 & \lm_2
                                               \end{pmatrix}$ 
and $\scr{A}_2 = \begin{pmatrix}
                  s_{21} & t_2 \\
                  t_2 & s_{22}
                  \end{pmatrix}$.}
\smallskip

Let $K_j= P_j \cap \ba{\ball(0;1)}$, $j=0,1,2.$ Therefore, we have 
\begin{align}
 K_1& \subset \{((\lm_1+i)x, (\lm_2+i)y)\in \CC~:~x,y \in \rea\}, \notag \\
K_2 & \subset \{((s_{21}+i)x+t_2y, t_2x+(s_{22}+i)y)\in \CC~:~ x,y \in \rea \}. \notag
\end{align}

We shall, in view of the condition that pairwise unions are locally polynomially convex at 
$0\in \CC$, use Kallin's lemma to show the polynomial convexity of $K_0 \cup K$, 
where $K:= K_1\cup K_2$. For that, consider the polynomial 
\[
 F(z,w)= z^2+w^2.
\]
Clearly, 
\begin{equation}\label{E:fk0}
 F(K_0)\subset \{z\in \cplx~:~ z\geq 0 \}.
\end{equation}
 For $(z,w)\in K_1$, we have 
\begin{equation}\label{E:lm1lm2}
 F(z,w)= F((\lm_1+i)x, (\lm_2+i)y)= (\lm_1^2-1)x^2+(\lm_2^2-1)y^2 + 2i(\lm_1x^2+\lm_2y^2),
\end{equation}
and, for $(z,w)\in K_2$, 
\begin{align}
 F(z,w)&=F((s_{21}+i)x+t_2y, t_2x+(s_{22}+i)y) \notag \\
&= (s_{21}^2+t_2^2-1)x^2+(s_{22}^2+t_2^2-1)y^2+ 2(s_{21}+s_{22})t_2xy \notag\\ 
& \qquad\qquad\qquad\qquad\qquad\qquad+ 2i(s_{21}x^2+s_{22}y^2+2t_2xy) \label{E:antdiagsame}.
\end{align}
By hypothesis, we have 
\begin{align}
 det(\scr{A}_1+\scr{A}_1^{\sf T})>0 & \impl \lm_1\lm_2>0 \label{E:deta1}, \\
det(\scr{A}_2+\scr{A}_2^{\sf T})>0 & \impl s_{21}s_{22}>t_2^2>0 \label{E:deta2}.
\end{align}

Hence, in view of  \eqref{E:deta1} and \eqref{E:deta2},  
equations~\eqref{E:lm1lm2} and \eqref{E:antdiagsame} give us 
\[
 F(K) \subset (\cplx\setminus \rea)\cup\{0\}.
\]
Hence, we get 
\begin{equation}\label{E:fk}
\hull{F(K_0)}\cap \hull{F(K)}= \{0\},
\end{equation}
and 
\begin{equation}\label{E:finv0}
 F^{-1}\{0\}\cap (K_0\cup K)=\{0\}.
\end{equation}

\noindent Therefore, from \eqref{E:fk0}, \eqref{E:fk} and \eqref{E:finv0}, all the conditions 
of Lemma~\ref{L:kallin} are satisfied. Hence, $K_0\cup K$ is polynomially 
convex. 
\smallskip

\noindent {\em {\bf (b)} When $\mathscr{A}_1= \begin{pmatrix}
                                             \lm_1 & 0 \\
                                               0 & \lm_2
                                               \end{pmatrix}$ 
and $\scr{A}_2 = \begin{pmatrix}
                  s_{21} & -t_2 \\
                  t_2 & s_{22}
                  \end{pmatrix}$.}
\smallskip

As above, let $K_j= P_j \cap \ba{\ball(0;1)}$, $j=0,1,2,$ and $K=K_1\cup K_2$. We also get $K_1$ to be 
the same as that in subcase {\bf (a)} and 
\[
K_2 \subset \{((s_{21}+i)x-t_2y, t_2x+(s_{22}+i)y)\in \CC~:~ x,y \in \rea \}.  
\]
Again, we consider the polynomial 
\[
 F(z,w)= z^2+w^2.
\]

\noindent When $(z,w)\in K_1$, $F(z,w)$ is as in equation \eqref{E:lm1lm2}, and, for $(z,w)\in K_2$, 
\begin{align}
 F(z,w)&= F((s_{21}+i)x-t_2y, t_2x+(s_{22}+i)y)\notag\\
 &=  (s_{21}^2+t_2^2-1)x^2+(s_{22}^2+t_2^2-1)y^2+ 
2(s_{22}-s_{21})t_2xy +2i(s_{21}x^2+s_{22}y^2). \label{E:antdiagdiff}
\end{align}
The inequality in \eqref{E:deta1} remains the same but \eqref{E:deta2} is 
replaced by: 
\begin{align}
det(\scr{A}_2+\scr{A}_2^{\sf T})>0 & \impl s_{21}s_{22}>0 \label{E:deta2not2}.
\end{align}

In view of equations \eqref{E:deta1} and \eqref{E:deta2not2}, the expressions in 
\eqref{E:fk0}, \eqref{E:lm1lm2} and \eqref{E:antdiagdiff} give
\[
 \hull{F(K_0)}\cap \hull{F(K)}=\{0\}, 
\]
and 
\[
 F^{-1}\{0\}\cap (K_0\cup K)=\{0\}.
\]
Therefore, all the conditions of Kallin's lemma are satisfied. Hence $K_0\cup K$ 
is polynomially convex.
\medskip

\noindent {\em {\bf Case II.} When $A_1$ has non-real eigenvalues.}
\smallskip

\noindent Since $\sm(A_1)\subset \cplx \setminus \rea$, the first part of $(ii)$ of 
our theorem follows from Lemma~\ref{L:norform3planes}, Part $(ii)$. By hypothesis, 
$det(\scr{A}_1+\scr{A}_1^{\sf T})>0$ and $\exists T_0\in GL(2,\rea)$ such that 
$det(\scr{A}_2(T_0)+\scr{A}_2(T_0)^{\sf T})>0$, where  
\[
\scr{A}_1=\begin{pmatrix}
          s_{1} & -t_1 \\
          t_1 & s_{1}
          \end{pmatrix}
\; \text{and} \;
\scr{A}_2:= \scr{A}_2(T_0) = \begin{pmatrix}
                             s_{21} & -t_2 \\
                              t_2 & s_{22}
                              \end{pmatrix}. 
\]

As in the previous case, let $K_j= P_j \cap \ba{\ball(0;1)}$, $j=0,1,2$. We have 
\begin{align}
K_0 & \subset \{(x,y)\in \CC~:~ x,y \in \rea \}, \notag\\
 K_1& \subset \{((s_1+i)x-t_1y, t_1x+(s_1+i)y)\in \CC~:~x,y \in \rea \}, \notag \\
K_2 & \subset \{((s_{21}+i)x-t_2y, t_2x+(s_{22}+i)y)\in \CC~:~ x,y \in \rea \}. \notag
\end{align}
We shall again use Kallin's lemma to show the polynomial convexity of 
$K_0\cup K_1 \cup K_2$. Consider the polynomial 
\[
 F(z,w)=z^2+w^2.
\]
When $(z,w)\in K_0$, $F(z,w)$ is as in \eqref{E:fk0}. 
For $(z,w)\in K_1$, we have 
\begin{align}
 F(z,w) & =F((s_1+i)x-t_1y, t_1x+(s_1+i)y) \notag\\
& = (s_1^2+t_1^2-1)(x^2+y^2) + 2is_1(x^2+y^2), \label{E:cpev} 
\end{align}
and for $(z,w)\in K_2$, $F(z,w)$ is as in equation~\eqref{E:antdiagdiff}. 
Let $K=K_1\cup K_2$. As before, from homogeneity of the totally-real planes and the hypothesis that 
the pairwise unions of the given totally-real planes are locally 
polynomially convex at the origin, $K$ is polynomially convex. By hypotheses, we get that 
\begin{align}
 det(\scr{A}_1+\scr{A}_1^{\sf T})>0 & \impl s_1^2>0 \label{E:deta1c}\\
det(\scr{A}_2+\scr{A}_2^{\sf T})>0 & \impl s_{21}s_{22}>0 \label{deta2c}.
\end{align}
 \smallskip

Hence, in view of \eqref{E:fk0}, \eqref{E:cpev} and \eqref{E:antdiagdiff}, we conclude that 
\[
 F(K_0) \subset \{z\in \cplx : z\geq 0\}; F(K)\subset (\cplx \setminus \rea)\cup \{0\},
\]
and
\[
 F^{-1}\{0\}\cap(K_0\cup K)= \{0\}.
\]
Therefore, all the conditions of Lemma~\ref{L:kallin} are satisfied. Hence, $K_0\cup K$ is 
polynomially convex.
\smallskip

It is now time to show that our conditions are optimal. We examine the one-parameter family of 
triples $(P_0^{\eps},P_1^{\eps},P_2^{\eps})$, where $P_0^{\eps}=\RR \; \forall \eps $ 
and $P_j^{\eps}$ are as determined by the matrices $A_j^{\eps}$, $j=1,2$, given in 
\eqref{E:PTmatrices}. From the discussion preceding \eqref{E:PTmatrices} we already know 
that pairwise unions of $P_0^{\eps},P_1^{\eps},P_2^{\eps}$ are locally polynomially convex 
at $0\in \CC$ and that $P_0^{\eps}\cup P_1^{\eps}\cup P_2^{\eps}$ $\forall \eps \in \rea \setminus \{0\}$, 
{\em for $\eps$ sufficiently small}, is not locally polynomially convex at $0\in \CC$. So we need to 
show that $(A_1^{\eps},A_2^{\eps})\in \OM\setminus \OM^*$ and 
$(A_1^{\eps},A_2^{\eps})\tends \bdy \OM^*$ as $\eps \tends 0$.
\smallskip

An elementary computation gives: 
\begin{align} 
\sm(A_1^{\eps}) &= \left\lbrace\frac{-\eps^2+ \sqrt{4\eps^2 -3}}{\sqrt{3}(1-\eps^2)}, \frac{-\eps^2-\sqrt{4\eps^2 -3}}
{\sqrt{3}(1-\eps^2)}\right\rbrace \notag \\
\sm(A_2^{\eps}) &= \left\lbrace \frac{\eps^2+ \sqrt{4\eps^2-3}}{\sqrt{3}(1-\eps^2)}, \frac{-\eps^2- 
\sqrt{4\eps^2 -3}}{\sqrt{3}(1-\eps^2)}\right\rbrace.\notag 
\end{align}
Clearly, for $\eps~:~0<\eps \ll 1$, $(A_1^{\eps},A_2^{\eps})\in \OM$. Now, from 
 $(ii)$ (read in the contrapositive) in the statement of Theorem~\ref{T:3planes} and Result~\ref{R:PTnpcvx}, 
we already know that $(A_1^{\eps},A_2^{\eps})\notin \OM^*$ $\forall \eps~:~
0<\eps \ll 1$. Hence, 
\[
(A_1^{\eps},A_2^{\eps})\in \OM \setminus \OM^* \; \forall \eps~:~0<\eps\ll 1. 
\]
Now observe that: 
\[
\lim_{\eps \tends 0} A_j^{\eps}= \begin{pmatrix}
                                                   0 & 1\\
                                                   -1 & 0
                                                  \end{pmatrix}  
=: A_j^0 , \; j=1,2.
\]
In the notation of the statement of Theorem~\ref{T:3planes},
 $\scr{A}_1^0=  \begin{pmatrix}
                 0 & 1\\
                 -1 & 0
                \end{pmatrix}.$
Define 
\[
\scr{S}:=\{T\in GL(2,\rea) : TA_1^0T^{-1}=\scr{A}_1^0\}. 
\]
Clearly, $\id \in \scr{S}$ 
and, in our notation: 
\[
 det(\scr{A}_2^0(\id)+\scr{A}_2^0(\id)^{\sf T}) = det(A_2^0+(A_2^0)^{\sf T})=0.
\]

Thus, appealing to the inequalities in Part $(ii)$ of our theorem that define $\OM^*$, 
we get $(A_1^0,A_2^0)\in \bdy \OM^*$.

\section{Proof of the Theorem~\ref{T:3planesevrl}}\label{Proof:3planesevrl}
\smallskip

\begin{proof}[Proof of Theorem~\ref{T:3planesevrl}]
\noindent $(i)$ First, we shall prove three claims. From these claims, the proof of Part $(i)$ 
Theorem~\ref{T:3planesevrl} will follow by appealing to Theorem~\ref{T:3planes}.
\smallskip

\noindent {\em {\bf Claim 1.} If $det[A_1,A_2]>0$, then $A_j$ cannot have non-real 
eigenvalues.}
\smallskip

\noindent {\em Proof of the Claim 1.} Suppose $A_1$ has non-real eigenvalues. Then, in view of 
Lemma~\ref{L:norform3planes}, there exists a $T\in GL(2, \rea)$ such that 
\[
 TA_1T^{-1}= \begin{pmatrix}
              s_1 & -t_1 \\
              t_1 & s_1
             \end{pmatrix} 
\; \text{and}\; 
TA_2T^{-1}= \begin{pmatrix}
             s_{21} & -t_2 \\
              t_2 & s_{22}
            \end{pmatrix}.
\]
Now, by a simple computation, we see that 
\[
 det[TA_1T^{-1}, TA_2T^{-1}]= -t_1^2(s_{22}-s_{21})^2 \leq0, 
\]
which is a contradiction to the fact that $det[A_1, A_2]>0$. Hence, $A_j$ 
cannot have non-real eigenvalues.
\smallskip

\noindent {\em {\bf Claim 2.} If $det[A_1,A_2]>0$, then each $A_j$ has distinct eigenvalues.}
\smallskip

\noindent {\em Proof of Claim 2.} Suppose $A_1$ does not have distinct eigenvalues. Let 
$\sm (A_1)=\{\lm_1\}$. By Claim 1, we have $\lm_1\in \rea$. Hence, there exists 
$T\in GL(2,\rea)$ such that 
\[
 TA_1T^{-1}= \begin{pmatrix}
              \lm_1 & \mu \\
              0 & \lm_1
             \end{pmatrix}. 
\]
Let us write $TA_2T^{-1}= TA_2T^{-1}= \begin{pmatrix}
                                          a & b \\
                                           c & d
                                       \end{pmatrix}.$
Then, again, by a simple computation, we see that 
\[
 det[TA_1T^{-1}, TA_2T^{-1}]= -c^2 \mu^2 \leq 0,
\]
which is a contradiction.
\smallskip

\noindent {\em {\bf Claim 3.} If $det[A_1,A_2]>0$, then there exists a $T \in GL(2, \rea)$ such that} 
\[
TA_1T^{-1}= 
\begin{pmatrix}
 \lm_1 & 0 \\
  0 & \lm_2
 \end{pmatrix},\; 
TA_2T^{-1}= \begin{pmatrix}
              s_{21} & t_1 \\
               t_1 & s_{22}
             \end{pmatrix}.
\]
 
\smallskip

\noindent {\em Proof of Claim 3.}  In view of Claim 1 and Claim 2, we conclude that $A_1$ has distinct 
eigenvalues. Hence, applying Lemma~\ref{L:norform3planes}, we get that there exists 
$T\in GL(2, \rea)$ such that 
\[
 TA_1T^{-1}= 
\begin{pmatrix}
 \lm_1 & 0 \\
  0 & \lm_2
 \end{pmatrix},\;\text{and}\; 
TA_2T^{-1}= \begin{pmatrix}
              s_{21} & t_1 \\
               t_1 & s_{22}
             \end{pmatrix}
\; \text{or}\; 
\begin{pmatrix}
              s_{21} & -t_1 \\
               t_1 & s_{22}
             \end{pmatrix}.
\]

Suppose $TA_2T^{-1}=\begin{pmatrix}
              s_{21} & -t_1 \\
               t_1 & s_{22}
             \end{pmatrix}$. Again calculating the commutator, we note that 
\[
 det[TA_1T^{-1}, TA_2T^{-1}]= -t_2^2(\lm_2 - \lm_1)^2 \leq 0,
\]
which is a contradiction. Hence the claim.

\medskip

We now resume the proof of Theorem~\ref{T:3planesevrl}. 
In view of Claim 3, we always get a $T \in GL(2, \rea)$ such that 
\[
TA_1T^{-1}= 
\begin{pmatrix}
 \lm_1 & 0 \\
  0 & \lm_2
 \end{pmatrix},\;\text{and}\; 
TA_2T^{-1}= \begin{pmatrix}
              s_{21} & t_1 \\
               t_1 & s_{22}
             \end{pmatrix}.
\]
We now observe that, in this case, 
\begin{align}
det(\scr{A}_1+\scr{A}_1^{\sf T}) &= 4detA_1, \notag\\
det(\scr{A}_2+\scr{A}_2^{\sf T})&= 4detA_2. \notag
\end{align} 
Hence, the conditions $detA_j>0$ imply that we can appeal Part $(i)$ of Theorem~
\ref{T:3planes}. Therefore, $P_0\cup P_1 \cup P_2$ is 
locally polynomially convex at the origin. 
\medskip

\noindent $(ii)$ We shall again use Kallin's lemma for the proof of this part. Before that, let us obtain 
simpler form of the matrices that will be used in the proof. 
\smallskip

\noindent {\bf Claim 4.} {\em It suffices to work with the union $\RR\cup M(\scr{A}_1)\cup M(\scr{A}_2)$, 
where:}  
\[
 \scr{A}_1= \begin{pmatrix}
             \lm_1 & 0 \\
              0 & \lm_2
            \end{pmatrix}
\; \text{and}\; \scr{A}_2 = 
\begin{pmatrix}
 s_{21} & -t_2 \\
 t_2 & s_{22}
\end{pmatrix}.
\]
\smallskip

\noindent {\em Proof of the Claim.} Since $detA_j<0$ for $j=1,2$, each $A_j$ must have real distinct eigenvalues. 
Hence, in view of Lemma~\ref{L:norform3planes}, we can find a $T\in GL(2, \rea)$ such that 
\[
TA_1T^{-1}=\begin{pmatrix}
                       \lm_1 & 0 \\
                       0 & \lm_2
                      \end{pmatrix}\;\; \text{and} \;\;
  TA_2T^{-1}= \begin{pmatrix}
                       s_{21} & t_2 \\
                       t_2 & s_{22}
                      \end{pmatrix}
\;\; \text{or} \;\;   \begin{pmatrix}
                       s_{21} & -t_2 \\
                       t_2 & s_{22}
                      \end{pmatrix}, 
\]
 for  $\lm_j,s_{2j}, t_2 \in \rea, \; j=1,2$. Let $\scr{A}_j=TA_jT^{-1}$ for $j=1,2$. 

Suppose $\scr{A}_2=\begin{pmatrix}
                       s_{21} & t_2 \\
                       t_2 & s_{22}
                      \end{pmatrix}$. Then, by a simple computation, we can see that 
\[
 det[A_1,A_2]=det[\scr{A}_1, \scr{A}_2]= (\lm_1-\lm_2)^2t_2^2 >0.
\]
This is a contradiction to the assumption that $det[A_1,A_2]<0$. Hence, 
\[
 \scr{A}_2 = \begin{pmatrix}
                       s_{21} & -t_2 \\
                       t_2 & s_{22}
                      \end{pmatrix}.
\]
The claim follows from Lemma~\ref{L:invertible}.
\smallskip

As before, to simplify notation, we shall denote $M(\scr{A}_j)$ as $P_j$, $j=1,2$.
As in the earlier cases, let $K_j= P_j \cap \ba{\ball(0;1)}$, $j=0,1,2$. We have 
\begin{align}
K_0 & \subset \{(x,y)\in \CC~:~ x,y \in \rea\}, \notag\\
 K_1& \subset \{((\lm_1+i)x, (\lm_2+i)y)\in \CC~:~x,y \in \rea \}, \notag \\
K_2 & \subset \{((s_{21}+i)x-t_2y, t_2x+(s_{22}+i)y)\in \CC~:~ x,y \in \rea \}. \notag
\end{align}
We now show that  
$K_0\cup K$ ia polynomially convex, where $K=K_1 \cup K_2$. Consider the polynomial 
\[
 F(z,w)=z^2-w^2.
\]
Clearly, 
\begin{equation}\label{E:fko3planesevrlii}
 F(K_0)\subset \subset \rea \subset \cplx.
\end{equation}
For $(z,w)\in K_1$, we have 
\begin{align}
 F(z,w)&= F((\lm_1+i)x, (\lm_2+i)y) \notag\\
       & = (\lm_1^2-1)x^2+(1-\lm_2^2)y^2+2i(\lm_1x^2-\lm_2y^2) \label{E:fk13planesevrlii},
\end{align}
and, for $(z,w)\in K_2$, 
\begin{align}
 F(z,w) &= F((S_{21}+i)x-t_2y, t_2x+(s_{22}+i)y) \notag \\
        &= (s_{21}^2-t_2^2-1)x^2+(1-s_{22}^2+t_2^2)y^2-2(s_{21}+s_{22})t_1xy \notag\\
&\qquad\qquad\qquad\qquad\qquad\qquad\qquad\qquad+ 2i(s_{21}x^2-s_{22}y^2-2t_2xy).\label{E:fk23planesevrlii} 
\end{align}
 We now show that $\hull{F(K)}\cap \hull{F(K_0)}=\{0\}$. We have 
\begin{align}
 detA_1<0 & \impl \lm_1\lm_2<0, \label{E:3planesevrldeta1} \\
detA_2<0 & \impl s_{21}s_{22}<-t_2^2<0. \label{E:3planesevrldeta2}
\end{align}
In view of \eqref{E:3planesevrldeta1} and \eqref{E:3planesevrldeta2}, expressions \eqref{E:fk13planesevrlii} 
and \eqref{E:fk23planesevrlii} give us 
\[
 F(K)\subset (\cplx \setminus \rea)\cup \{0\}.
\]
 Hence, we have
\[
\hull{F(K)}\cap \hull{F(K_0)}=\{0\}\; \text{and}\; F^{-1}\{0\}\cap K=\{0\}.
\]
We also have 
\[
 F^{-1}\{0\}\cap K_0= \{(x,y)\in K_0~:~x=\pm y\}.
\]
Hence, $F^{-1}\{0\}\cap(K\cup K_0)$ is polynomially convex.
Therefore, all the conditions of Lemma~\ref{L:kallin} are satisfied. Hence, $K\cup K_0$ is 
polynomially convex.
\end{proof}  
\smallskip
        
{\bf Acknowledgement.} I would like to thank Gautam Bharali for many useful comments on this paper.

\end{document}